\documentclass[reqno, 11pt, psamsfonts]{amsart}
\usepackage{mathabx}    
\usepackage{mathrsfs}  
\usepackage{xcolor}  	
\usepackage[backref]{hyperref}
\hypersetup{
  colorlinks,
  linkcolor={red!60!black},
  citecolor={green!60!black},
  urlcolor={blue!60!black},
}

\usepackage{doi}

\usepackage{amsmath,amssymb,amsthm}
\usepackage[abbrev,msc-links]{amsrefs}   
\usepackage[utf8]{inputenc}
\usepackage{tikz}
\usepackage{amsmath}
\usepackage{comment}
\usepackage{datetime}
\theoremstyle{plain}
\newtheorem{thm}{Theorem}[section]

\newtheorem{prop}[thm]{Proposition}

\newtheorem{lemma}[thm]{Lemma}

\newtheorem{question}[thm]{Question}

\theoremstyle{definition}

\newtheorem{dfn}[thm]{Definition}

\usepackage{fullpage}
\usepackage{setspace}
\let\theta=\vartheta
\let\rho=\varrho
\let\phi=\varphi

\def\GG{\mathbb G}
 
\def\cA{{\mathcal A}}

\def\cF{{\mathcal F}}

\def\cQ{{\mathcal Q}}

\def\cS{{\mathcal S}}
\def\cT{{\mathcal T}}

\def\var{\mathop{\text{\rm Var}}\nolimits}
\def\be{\mathop{\text{\rm Be}}\nolimits}
\def\abs{\mathop{\text{\rm abs}}\nolimits}

\def\({\left(}
\def\){\right)}  

\def\:{\colon}

\def\aas{{a.a.s.}}
\def\th{${}^{\text{th}}$}

\usepackage{enumitem}
\def\rmlabel{\upshape({\itshape \roman*\,})}

\let\polishlcross=\l
\def\l{\ifmmode\ell\else\polishlcross\fi}

\def\rmlabel{\upshape({\itshape \roman*\,})}

\def\qcomma{\text{,}\quad}

\def\section{\@startsection{section}{1}%
  \z@{.7\linespacing\@plus\linespacing}{.5\linespacing}%
  {\normalfont\scshape\centering}}

\makeatletter
\renewcommand\section{%
	\@ifstar{%
		\@startsection{section}{1}%
  		\z@{.7\linespacing\@plus\linespacing}{.5\linespacing}%
  		{\normalfont\scshape\centering%
  		}}{%
		\@startsection{section}{1}%
  		\z@{.7\linespacing\@plus\linespacing}{.5\linespacing}%
  		{\normalfont\scshape\centering%
		\S %add the ß-symbol in non-starred version
  		}}}                   
\makeatother

\theoremstyle{plain}
\let\epsilon=\varepsilon

\newtheoremstyle{note}% name
  {4pt}%      Space above
  {4pt}%      Space below 
  {\sl}%      Body font
  {}%         Indent amount (empty = no indent, \parindent = para indent)
  {\itshape}% Thm head font
  {.}%        Punctuation after thm head
  {.5em}%     Space after thm head: " " = normal interword space;
  {}%         Thm head spec (can be left empty, meaning `normal')

\theoremstyle{note}

\usepackage{lineno}
\newcommand*\patchAmsMathEnvironmentForLineno[1]{%
\expandafter\let\csname old#1\expandafter\endcsname\csname #1\endcsname
\expandafter\let\csname oldend#1\expandafter\endcsname\csname end#1\endcsname
\renewenvironment{#1}%
{\linenomath\csname old#1\endcsname}%
{\csname oldend#1\endcsname\endlinenomath}}% 
\newcommand*\patchBothAmsMathEnvironmentsForLineno[1]{%
\patchAmsMathEnvironmentForLineno{#1}%
\patchAmsMathEnvironmentForLineno{#1*}}%
\AtBeginDocument{%
\patchBothAmsMathEnvironmentsForLineno{equation}%
\patchBothAmsMathEnvironmentsForLineno{align}%
\patchBothAmsMathEnvironmentsForLineno{flalign}%
\patchBothAmsMathEnvironmentsForLineno{alignat}%
\patchBothAmsMathEnvironmentsForLineno{gather}%
\patchBothAmsMathEnvironmentsForLineno{multline}%
}

\begin{document}
\onehalfspace
%\linenumbers
\footskip=28pt
\shortdate
\yyyymmdddate
\settimeformat{ampmtime}
\date{\today, \currenttime}

\title[Powers of Hamilton cycles in perturbed hypergraphs]%
{Powers of tight Hamilton cycles in randomly perturbed hypergraphs}

\author[W. Bedenknecht]{Wiebke Bedenknecht}
\author[J. Han]{Jie Han}
\author[Y. Kohayakawa]{Yoshiharu Kohayakawa}
\author[G. O. Mota]{Guilherme Oliveira Mota}
\address{Fachbereich Mathematik, Universit\"at Hamburg,
  Bundesstra\ss{}e~55, D-20146 Hamburg, Germany}
\email{Wiebke.Bedenknecht@uni-hamburg.de}

\address{Instituto de Matem\'atica e Estat\'{\i}stica, Universidade de
	S\~ao Paulo, S\~ao Paulo, Brazil}
\email{\{jhan|yoshi\}@ime.usp.br}

\address{Centro de Matem\'atica, Computa\c c\~ao e Cogni\c c\~ao, Universidade Federal do ABC, Santo Andr\'e, Brazil}
\email{g.mota@ufabc.edu.br}

\thanks{%
  The second author was supported by FAPESP (Proc.~2014/18641-5).  The
  third author was supported by FAPESP (Proc.~2013/03447-6) and by
  CNPq (Proc.~459335/2014-6, 310974/2013-5).  The fourth author was
  supported by FAPESP (Proc.~2013/11431-2, 2013/20733-2).
  The collaboration of the authors was supported by the CAPES/DAAD
  PROBRAL programme (Proj.~430/15, 57143515).
}

\subjclass[2010]{%05C35 (primary), 05C65, 05C80 (secondary)
}
\begin{abstract}
  For $k\ge 2$ and $r\ge 1$ such that $k+r\ge 4$, we prove that, for
  any $\alpha>0$, there exists $\epsilon>0$ such that the union of an
  $n$-vertex $k$-graph with minimum codegree
  $\(1-\binom{k+r-2}{k-1}^{-1}+\alpha\)n$
  and a binomial random $k$-graph $\mathbb{G}^{(k)}(n,p)$ with
  $p\ge n^{-\binom{k+r-2}{k-1}^{-1}-\epsilon}$ on the same vertex set
  contains the $r$\th{} power of a tight Hamilton cycle with high
  probability.  This result for $r=1$ was first proved by McDowell and
  Mycroft.
%[Hamilton $\ell$-cycles in random perturbed hypergraphs,  \textit{submitted}].
\end{abstract} 

\maketitle

%\tableofcontents

\section{Introduction}
\label{sec:intro}

\subsection{Hamiltonian cycles}
The study of Hamiltonicity (the existence of a cycle as a spanning
subgraph) has been a central and fruitful area in graph theory.  In
particular, by a celebrated result of Karp~\cite{Karp}, the decision
problem for Hamiltonicity in general graphs is known to be
NP-complete.  Therefore it is likely that good characterizations of
graphs with Hamilton cycles do not exist, and it becomes natural to study
sufficient conditions that guarantee Hamiltonicity.  Among a large
variety of such results, the most famous one is the classical theorem
of Dirac from~1952: every $n$-vertex graph ($n\ge 3$) with minimum
degree at least $n/2$ is Hamiltonian~\cite{Di52}.

Another well-studied object in graph theory is the binomial random
graph $\mathbb{G}(n,p)$, which contains $n$ vertices and each pair of
vertices forms an edge with probability $p$ independently from all
other pairs.  P\'osa~\cite{Posa} and Korshunov~\cite{Korshunov}
independently determined the threshold for Hamiltonicity in
$\mathbb{G}(n,p)$, which is~$(\log n)/n$.  This implies that almost
all dense graphs are Hamiltonian.  In this sense the degree constraint
in Dirac's theorem is very strong.  In fact, Bohman, Frieze and
Martin~\cite{BFM} studied the random graph model that starts with a
given, dense graph and adds $m$ random edges.  In particular, they
showed that for every $\alpha>0$ there is $c=c(\alpha)$ such that if
we start with a graph with minimum degree at least~$\alpha n$ and we
add~$cn$ random edges, then the resulting graph is Hamiltonian a.a.s.\
(as usual, we say that an event happens \emph{asymptotically almost
  surely}, or a.a.s., if it happens with probability tending to $1$
as~$n\to\infty$).  By considering the complete bipartite graph with
vertex classes of sizes~$\alpha n$ and~$(1-\alpha)n$, one sees that the
result above is tight up to the value of~$c$.
%A comparison can be drawn to the notion of smoothed analysis of algorithms introduced by Spielman and Teng~\cite{SpTe}, which involves studying the performance of algorithms on randomly perturbed inputs.

%\subsection{Uniform hypergraphs}
It is natural to study Hamiltonicity problems in uniform hypergraphs.
Given $k\ge 2$, a \emph{$k$-uniform hypergraph} (in short, \emph{$k$-graph})
$H=(V,E)$ consists of a vertex set $V$ and an edge set $E\subseteq
\binom{V}{k}$; thus, every edge of~$H$  is a $k$-element subset of~$V$. 
Given a $k$-graph $H$ with a set $S$ of $d$ vertices (where $1 \le d
\le k-1$) we define $N_{H} (S)$ to be the collection of $(k-d)$-sets
$T$ such that $S\cup T\in E(H)$, and let $\deg_H(S):=|N_H(S)|$ (the
subscript $H$ is omitted whenever~$H$ is clear from the context). The
\emph{minimum $d$-degree $\delta _{d} (H)$} of $H$ is the minimum of
$\deg_{H} (S)$ over all $d$-vertex sets $S$ in $H$.   
We refer to $\delta _{k-1}(H)$ as the \emph{minimum codegree} of~$H$.

In the last two decades, there has been growing interest in extending
Dirac's theorem to $k$-graphs.  Among other notions of cycles in
$k$-graphs (e.g., Berge cycles), the following `uniform' cycles have
attracted much attention.  For integers $1\le \ell \le k-1$ and
$m\ge 3$, a $k$-graph $F$ with $m(k-\ell)$ vertices and $m$ edges is a
called an \emph{$\ell$-cycle} if its vertices can be ordered
cyclically so that each of its edges consists of $k$ consecutive
vertices and every two consecutive edges (in the natural order of the
edges) share exactly $\ell$ vertices.  Usually $(k-1)$-cycles are also
referred to as \emph{tight} cycles.
%Furthermore, if the vertex ordering is not cyclic, then we say that $F$ is an $\ell$-path.
%We refer to the first $\ell$ vertices and the last $\ell$ vertices as the \emph{$\ell$-ends} of the $\ell$-path.
%We write $P_m$ for an $\ell$-path on $m$ edges.
%In $k$-graphs, a $(k-1)$-cycle is often called a \emph{tight} cycle. 
We say that a $k$-graph contains a \emph{Hamilton $\ell$-cycle} if it
contains an $\ell$-cycle as a spanning subgraph. 
%Note that a Hamilton $\ell$-cycle of a $k$-graph on $n$ vertices contains exactly $n/(k - \ell)$ edges, implying that $k- \ell$ divides $n$.
In view of Dirac's theorem, minimum $d$-degree conditions that
force Hamilton $\ell$-cycles (for $1\le d,\,\ell\le k-1$) have been
studied intensively~\cite{BMSSS1, BMSSS2,  BHS, CzMo, GPW, HS, HZ2,
  HZ1, KKMO, KMO, KO, RRRSS, RoRu14, RoRuSz06, RRS08, RRS11}. 

Let $\mathbb{G}^{(k)}(n,p)$ denote the binomial random $k$-graph on
$n$ vertices, where each $k$-tuple forms an edge independently with
probability $p$.  The threshold for the existence of Hamilton
$\ell$-cycles has been studied by Dudek and Frieze~\cite{DuFr1,
  DuFr2}, who proved that for $\ell=1$ the threshold is
$(\log n)/n^{k-1}$, and for $\ell\ge 2$ the threshold is
$1/n^{k-\ell}$ (they also determined sharp thresholds for
every~$k\ge 4$ and $\ell=k-1$).

%It is also natural to consider the random perturbed $k$-graphs.
Krivelevich, Kwan and Sudakov~\cite{KKS} considered randomly perturbed
$k$-graphs, which are $k$-graphs obtained by adding random edges to a
fixed $k$-graph.  They proved the following theorem, which mirrors the
result of Bohman, Frieze and Martin~\cite{BFM} for randomly perturbed
graphs mentioned earlier.

\begin{thm}\label{thm:KKS}\cite{KKS}
  For any~$k\geq2$ and~$\alpha>0$, there is~$c_k=c_k(\alpha)$ for
  which the following holds.  Let~$H$ be a $k$-graph on
  $n\in(k-1)\mathbb{N}$ vertices with $\delta_{k-1}(H)\geq \alpha n$.
  If $p=c_k n^{-(k-1)}$, then the union $H\cup \mathbb{G}^{(k)}(n,p)$
  a.a.s.\ contains a Hamilton $1$-cycle.
\end{thm}

The authors of~\cite{KKS} also obtained a similar result for perfect
matchings.  These results are tight up to the value of~$c_k$, as shown
by a simple `bipartite' construction.  McDowell and
Mycroft~\cite{McMy} and, subsequently, Han and Zhao~\cite{HZ_pert} extended Theorem~\ref{thm:KKS} to Hamilton
$\ell$-cycles and other degree conditions.

\subsection{Powers of Hamilton cycles}
Powers of cycles are natural generalizations of cycles.  Given
$k\ge 2$ and~$r\ge 1$, we say that a $k$-graph with~$m$ vertices is an
\emph{$r$\th{} power of a tight cycle} if its
vertices can be ordered cyclically so that each consecutive $k+r-1$
vertices span a copy of~$K_{k+r-1}^{(k)}$, the complete $k$-graph on
$k+r-1$ vertices, and there are no other edges than the ones forced by
this condition.  This extends the notion of (tight) cycles in
hypergraphs, which corresponds to the case $r=1$.

The existence of powers of paths and cycles has also been intensively
studied.  For example, the famous P\'osa--Seymour conjecture, which
was proved by Koml\'os, S\'ark\"ozy and
Szemer\'edi~\cite{KoSaSz98a,KoSaSz98b} for sufficiently large graphs,
states that every $n$-vertex graph with minimum degree at least
$r n/(r+1)$ contains the $r$\th{} power of a Hamilton cycle.  A general
result of Riordan~\cite{Ri00} implies that, for $r\geq 3$, the
threshold for the existence of the $r$\th{} power of a Hamilton cycle in
$\mathbb{G}(n,p)$ is $n^{-1/r}$.  The case~$r=2$ was investigated by
K\"uhn and Osthus~\cite{KuOs12}, who proved that
$p\ge n^{-1/2+\epsilon}$ suffices for the existence of the square of a
Hamilton cycle in $\mathbb{G}(n,p)$, which is sharp up to
the~$n^\epsilon$ factor.  This was further sharpened by Nenadov and
{\v S}kori{\'c}~\cite{NeSk16}.
Moreover, Bennett, Dudek and Frieze~\cite{BeDuFr16} proved a result
for the square of a Hamilton cycle in randomly perturbed graphs,
extending the result of Bohman, Frieze and Martin~\cite{BFM}.

\begin{thm}\label{thm:BDF}\cite{BeDuFr16}
  For any $\alpha>0$ there is~$K>0$ such that the following holds.
  Let~$G$ be a graph with $\delta(G)\ge (1/2+\alpha) n$ and suppose
  $p=p(n)\ge K n^{-2/3}\log^{1/3}n$.  Then the union
  $H\cup \mathbb{G}(n,p)$ a.a.s.\ contains the square of a Hamilton
  cycle.
\end{thm}

Note that in Theorem~\ref{thm:BDF} the randomness that is required is much weaker than the one needed in the result for the pure random model (which is essentially $n^{-1/2}$). The authors of~\cite{BeDuFr16} also asked for similar results for higher powers of Hamilton cycles in randomly perturbed graphs.

Parczyk and Person~\cite[Theorem~3.7]{PP} proved that, for $k\geq 3$
and $r\geq 2$, the threshold for the existence of an $r$\th{} power of a
tight Hamilton cycle in the random $k$-graph $\mathbb{G}^{(k)}(n,p)$
is $n^{-\binom{k+r-2}{k-1}^{-1}}$.  Our main result,
Theorem~\ref{main} below, shows that if we consider randomly perturbed
$k$-graphs $H\cup \mathbb{G}^{(k)}(n,p)$ with $\delta_{k-1}(H)$
reasonably large, then
$p=p(n)\ge n^{-\binom{k+r-2}{k-1}^{-1}-\epsilon}$ is enough to
guarantee the existence of an $r$\th{} power of a tight Hamilton cycle
with high probability.

\begin{thm}[Main result]\label{main}
  For all integers $k\ge 2$ and $r\ge 1$ such that $k+r\ge 4$ and
  $\alpha>0$, there is $\epsilon>0$ such that the following holds.
  Suppose~$H$ is a $k$-graph on $n$ vertices with
  \begin{equation}
    \label{eq:main_min_deg}
    \delta_{k-1}(H)\ge \left( 1- \binom{k+r-2}{k-1}^{-1} + \alpha
    \right) {n}
  \end{equation}
  and $p=p(n)\ge n^{-\binom{k+r-2}{k-1}^{-1}-\epsilon}$.  Then~\aas~the
  union $H\cup \mathbb{G}^{(k)}(n,p)$ contains the $r$\th{} power of
  a tight Hamilton cycle.
\end{thm}

We remark that our proof only gives a small $\epsilon$, and it
would be interesting to know if one can get a larger gap in comparison
with the result in the purely random model, as in
Theorem~\ref{thm:BDF}.  We remark that the case $k\ge 3$ and $r=1$ of
Theorem~\ref{main} was first proved by McDowell and
Mycroft~\cite{McMy}.
Other results in randomly perturbed graphs can be found in~\cite{BTW, BMPP, KKS2, BHKMPP, HZ_pert}.

The core of the proof of Theorem~\ref{main} follows the \textit{Absorbing Method} introduced by R\"odl, Ruci\'nski, and Szemer\'edi in~\cite{RoRuSz06}, combined with results concerning binomial random hypergraphs.

This paper is organized as follows.  In Section~\ref{sec:random} we
prove some results concerning random hypergraphs.
Section~\ref{sec:abs-con} contains two essential lemmas in our
approach, namely, Lemma~\ref{lm:conn} (Connecting Lemma) and
Lemma~\ref{lm:abs} (Absorbing Lemma).  In Section~\ref{main} we prove
our main result, Theorem~\ref{main}.  Some remarks concerning the
hypotheses in Theorem~\ref{main} are given in
Section~\ref{sec:concluding}.  Throughout the paper, we omit floor and
ceiling functions.

\section{Subgraphs of random hypergraphs}\label{sec:random}

In this section we prove some results related to binomial random
$k$-graphs.  We will apply Chebyshev's inequality and Janson's inequality to prove some concentration results that we shall
need.
 For convenience, we state these two inequalities in the form we need (inequalities~\eqref{eq:2} and~\eqref{eq:1} below follow, respectively, from Janson's and Chebyshev's inequalities; see,
e.g.,~\cite[Theorem~2.14]{JLR}).

We first recall Janson's inequality.
Let $\Gamma$ be a finite set and let $\Gamma_p$ be a random subset of $\Gamma$ such that each element of $\Gamma$ is included in $\Gamma_p$ independently with probability $p$.
Let $\cS$ be a family of non-empty subsets of $\Gamma$ and for each $S\in \cS$, let $I_S$ be the indicator random variable for the event $S\subseteq \Gamma_p$.
%$I_A = \textbf{1}[A\subseteq \Gamma_p]$.
Thus each $I_S$ is a Bernoulli random variable $\be(p^{|S|})$.
%Given a family of random variables $\{I_i\}_{i\in \mathcal{I}}$, for $i, j\in \mathcal I$, we write $i\sim j$ if and only if $I_i$ and $I_j$ are dependent. 
%Let $\Delta = \sum_{i\sim j}\mathbb{E}[I_iI_j]$, where the sum is over unordered pairs.
Let $X:=\sum_{S\in \cS} I_S$ and $\lambda := \mathbb E(X)$.
Let $\Delta_{X} := \sum_{S\cap T\neq\emptyset}\mathbb{E}(I_S I_T)$, where the sum is over all ordered pairs $S, T\in \cS$ (note that the sum includes the pairs $(S,S)$ with $S\in \cS$).
Then Janson's inequality says that, for any $0\le t\le \lambda$,
\begin{equation}
\mathbb{P}(X\leq \lambda -t)\leq \exp \left ( -\dfrac{t^2}{2\Delta_{X}}\right ). \label{eq:2}
\end{equation}
Next note that $\var(X)=\mathbb E(X^2) - \mathbb E(X)^2 \le \Delta_X$.
Then, by Chebyshev's inequality,
\begin{equation}
\mathbb{P}(X\ge 2\lambda) \le \frac{\var(X)}{\lambda^2} \le \frac{\Delta_X}{\lambda^2}. \label{eq:1}
\end{equation}

%Note that $I_A\in \be(p^{f})$ and define $\Delta_F: =  \sum_{A\cap B\neq\emptyset}\mathbb{E}[I_A I_B]$, where the sum is over all distinct $s$-sets $A, B\subset V$. %such that $A\cap B\neq \emptyset$.
%The following proposition for counting subgraphs in $\mathbb G^{(k)}(n,p)$ is standard. (?)
Consider the random $k$-graph $\mathbb{G}^{(k)}(n, p)$ on an $n$-vertex set $V$.
Note that we can view $\mathbb{G}^{(k)}(n, p)$ as $\Gamma_p$ with $\Gamma = \binom{V}k$.
For two $k$-graphs $G$ and $H$, let $G\cap H$ (or $G\cup H$) denote the $k$-graph with vertex set $V(G)\cap V(H)$ (or $V(G)\cup V(H)$) and edge set $E(G)\cap E(H)$ (or $E(G)\cup E(H)$).
Finally, let
\begin{equation*}
\Phi_F = \Phi_F(n, p)= \min \{n^{v_H} p^{e_H}: H\subseteq F\text{ and } e_H>0\}.
\end{equation*}
The following simple proposition is useful.

\begin{prop}\label{prop:1}%\cite{BHKM}
  Let $F$ be a $k$-graph with $s$ vertices and $f$ edges and let
  $G:=\mathbb{G}^{(k)}(n, p)$.  Let $\cA$ be a family of ordered
  $s$-subsets of $V=V(G)$.  For each~$A\in\cA$, let~$I_A$ be the
  indicator random variable of the event that $A$ spans a labelled
  copy of~$F$ in~$G$.  Let $X=\sum_{A\in \cA} I_A$.  Then
  $\Delta_{X} \leq s! 2^{2s} n^{2s}p^{2f}/\Phi_F$.
\end{prop}

\begin{proof}
For each ordered $s$-subset $A$ of $V$, let $\alpha_A$ be the
bijection from $V(F)$ to $A$ following the orders of $V(F)$ and $A$.  Let~$F_A$ be the labelled copy of $F$ spanned on $A$.
For any $T\subseteq V(F)$ with $|F[T]|>0$, denote by $W_{T}$ the set of all pairs $A,\,B\in \cA$ such that $A\cap B=\alpha_A(T)$.
% and $f':=e_{F_A\cap F_B}>0$. 
If~$T$ has $s'$ vertices and $F[T]$ has $f'$ edges, then for every $\{A, B\}\in W_{T}$, $F_A\cup F_B$ has exactly $2s-s'$ vertices and at least $2f-f'$ edges. %where $f'=E(F_A[T])$.
Therefore, we can bound $\Delta_{X}$ by
%\begin{equation}\label{eq:sum}
\[
\Delta_{X} \le \sum_{T\subseteq V(F)} |W_{T}| p^{2f - f'}\, .
\]
%\end{equation}

%Now we note that $m$ edges determine at least $\ell+m(k-\ell)$ vertices.
Given integers $n$ and $b$, let $(n)_b:=n(n-1)(n-2)\cdots (n-b+1) = n!/(n-b)!$.
Note that there are at most $\binom{n}{2s-s'}$ choices for the vertex set of $F_A\cup F_B$, and there are at most
\[
(2s-s')_s\cdot\binom{s}{s'}s!\le (2s-s')! s!2^{s}
\]
ways to label each $(2s-s')$-set to get $\{A, B\}$.
Thus we have $|W_{T}|\le s! 2^{s}n^{2s-s'}$ and
\[
 \Delta_{X} \le \sum_{T\subseteq V(F)} s! 2^{s}n^{2s-s'} p^{2f - f'}\le \sum_{T\subseteq V(F)} s!2^{s}n^{2s} p^{2f}/\Phi_F \leq s! 2^{2s} n^{2s}p^{2f}/\Phi_F,
\]
because there are at most $2^s$  choices for $T$. 
\end{proof}
%%%%%%%%%%%%%%%%

The following lemma gives the properties of $\mathbb{G}^{(k)}(n,p)$
that we will use.  Throughout the rest of the paper, we write
$\alpha \ll \beta \ll \gamma$ to mean that `we can choose the positive
constants $\alpha$, $\beta$ and~$\gamma$ from right to left'.  More
precisely, there are functions $f$ and $g$ such that, given~$\gamma$,
whenever $\beta \leq f(\gamma)$ and $\alpha \leq g(\beta)$, the
subsequent statement holds.  Hierarchies of other lengths are defined
similarly.
%Moreover, when we use variables in their reciprocal form in the hierarchy, we implicitly assume that the variables are integers.

\begin{lemma}\label{lm:gnp}
  Let $F$ be a labelled $k$-graph with $b$ vertices and $a$ edges.
  Suppose $1/n\ll 1/C \ll \gamma, 1/a, 1/b, 1/s$.  Let $V$ be an
  $n$-vertex set, and let $\cF_1, \dots, \cF_t$ be $t\le n^{s}$
  families of $\gamma n^{b}$ ordered $b$-sets on $V$.  If $p=p(n)$ is
  such that $\Phi_{F}(n,p) \ge C n$, then the following properties
  hold for the binomial random $k$-graph $G=\mathbb{G}^{(k)}(n,p)$
  on~$V$.
\begin{enumerate}[label=\rmlabel]
\item With probability at least $1-\exp(-n)$, every induced subgraph of $G$ of order $\gamma n$ contains a copy of $F$.\label{item-I}
\item With probability at least $1-\exp(-n)$, for every $i\in [t]$, there are at least $(\gamma/2) n^{b}p^{a}$ ordered $b$-sets in $\cF_i$ that span labelled copies of $F$.\label{item-II}
\item With probability at least $1-1/\sqrt n$, there are at most $2 n^{b} p^{a}$ ordered $b$-sets of vertices of $G$ that span labelled copies of $F$.\label{item-III}
\item With probability at least $1-1/{\sqrt{ n}}$, the number of overlapping (i.e., not vertex-disjoint) pairs of copies of $F$ in $G$ is at most $4b^2 n^{2b-1} p^{2a}$.\label{item-IV}
\end{enumerate}
\end{lemma}
\begin{proof}%[Proof of Lemma~\ref{lm:gnp}]
  Let $\cA$ be a family of ordered $b$-sets of vertices in~$V$.  For
  each $A\in \cA$, let~$I_A$ be the indicator random variable of the
  event that $A$ spans a labelled copy of~$F$ in~$G$.  Let
  $X_\cA=\sum_{A\in \cA} I_A$.  From the hypothesis
  that~$\Phi_{F} \ge C n$ and Proposition~\ref{prop:1}, we have
\begin{equation}\label{eq:delta}
\Delta_{X} \leq b! 2^{2b} n^{2b}p^{2a}/\Phi_{F}\le b! 2^{2b} n^{2b}p^{2a}/(C n).
\end{equation}
Furthermore, let $\cS$ consist of the edge sets of the labelled copies of $F$ spanned on $A$ in the complete $k$-graph on $V$ for all $A\in \cA$.
Since we can write $X_\cA=\sum_{S\in \cS} I_S$, where $I_S$ is the indicator variable for the event $S\subseteq E(G)$, we can apply~\eqref{eq:2} to $X_\cA$.

For~\ref{item-I}, fix a vertex set $W$ of $G$ with $|W|=\gamma n$.
Let $\cA$ be the family of all labelled $b$-sets in $W$.
Let $X_\cA$ be the random variable that counts the number of members of $\cA$ that span a labelled copy of $F$ and thus $\mathbb{E}[X_\cA]= (\gamma n)_{b} p^a$.
%Note that $\delta:=\max_{i}\sum_{j\sim i} p^{a} \le n^{b-k}p^{a}$.
By~\eqref{eq:delta} and~\eqref{eq:2} and the fact that $1/C\ll \gamma, 1/b$, we have 
$\mathbb{P} (X_\cA=0) \le \exp( -2 n)$.
By the union bound, the probability that there exists a vertex set $W$ of size $\gamma n$ such that $X_\cA=0$ is at most
$2^n \exp(-2 n) \le \exp(-n)$,
which proves~\ref{item-I}. 

For~\ref{item-II}, fix $i\in [t]$ and let $X_{\cF_i}$ be the random variable that counts the members of $\cF_i$ that span $F$. 
%We will apply Lemma~\ref{lm:1}. 
Note that $\mathbb{E}[X_{\cF_i}]=\gamma n^{b}p^a$.
% and $\delta \le n^{b-k}p^{a}$.
Thus~\eqref{eq:2} implies  that
$\mathbb{P}\big(X_{\cF_i}\leq (\gamma/2)n^{b}p^{a}\big)\le \exp(- 2n)$.
By the union bound and the fact that $n^{s} \exp(- 2n) \le \exp(- n)$,
we see that~\ref{item-II} holds.

For~\ref{item-III}, let $X_3$ be the random variable that counts the number of labelled copies of~$F$ in~$G$. 
Since $\mathbb{E}(X_3)= (n)_{b}p^{a}$, by~\eqref{eq:delta} and~\eqref{eq:1}, we obtain 
\[
\mathbb{P}(X_3\ge 2p^an^b)\leq \mathbb{P}(X_3\ge 2\mathbb{E}[X_3])\le \frac{\Delta_{X_3}}{\mathbb{E}[X_3]^2} \le \frac{ b! 2^{2b} n^{2b}p^{2a}/(C n) }{( (n)_{b}p^{a})^2} \le \frac1{\sqrt n}.
\]
For~\ref{item-IV}, let $Y$ be the random variable that denotes the number of overlapping pairs of copies of~$F$ in~$G$.
We first estimate $\mathbb{E}[Y]$.
We write $Y=\sum_{A\in \cQ}I_A$, where $\cQ$ is the collection of the edge sets of overlapping pairs of labelled copies of $F$ in the complete $k$-graph on $n$ vertices. 
Note that if two overlapping copies of $F$ do not share any edge, then they induce at most $2b-1$ vertices and exactly $2a$ edges.
%in addition, if they induce $2b-1$ vertices, then there are at most $\frac{(2b-1)!}{(b-1)!} b\cdot b! = (2b-1)!b^2$ ways to label the vertices of these two copies.
Note that for $1\le i\le b$, there are 
\[
\binom{n}{2b-i} (2b-i)_{b}\binom{b}{i} b! = (n)_{2b - i} \binom b i (b)_i \le (n)_{2b - i}(b)_i ^2
\]
members of $\cQ$ whose two copies of $F$ share exactly $i$ vertices. 
Thus, the number of choices for the vertex sets of pairs of copies which induce at most $2b-2$ vertices is at most $\sum_{2\le i\le b} (n)_{2b - i} (b)_i^2 \le n^{2b-1}$.
By the definition of $\Delta_{X_3}$ and \eqref{eq:delta} we have
\[
{n}^{2b - 1} b^2 p^{2a}/2\le \mathbb{E}[Y]\le (n)_{2b - 1} b^2 \cdot p^{2a} + n^{2b-1} \cdot p^{2a} + \Delta_{X_3} \le 2 b^2 {n}^{{2b - 1}} p^{2a}.
\]

We next compute $\Delta_Y$. 
For each $A\in \cQ$, let $S_A$ denote the $k$-graph induced by $A$ (thus $S_A$ is the union of two overlapping copies of $F$).
For each $A, B\in \cQ$, write $S_A:=F_1\cup F_2$ and $S_B:=F_3\cup F_4$, where each $F_i$ is a copy of $F$ for $i\in [4]$ such that $E(F_1)\cap E(F_3)\neq \emptyset$.
Define $H_1:=F_1\cap F_2$, $H_2:=(F_1\cup F_2)\cap F_3$ and $H_3:=(F_1\cup F_2\cup F_3)\cap F_4$.
Since $V(F_1)\cap V(F_2)\neq\emptyset $, $V(F_3)\cap V(F_4)\neq\emptyset $, and $E(F_1)\cap E(F_3)\neq \emptyset$, we know that $v_{H_i}\ge 1$ for $i=1,2,3$.
We claim that $n^{v_{H_i}}p^{e_{H_i}}\ge n$ for $i=1,2,3$.
Indeed, since each $H_i$ is a subgraph of $F$, if $e_{H_i}\ge 1$, then $n^{v_{H_i}}p^{e_{H_i}}\ge \Phi_{F}\ge C n$; otherwise $e_{H_i}=0$ and then we have  $n^{v_{H_i}}p^{e_{H_i}} = n^{v_{H_i}} \ge n^1=n$.
%, and thus $n^{v_{H_1}}p^{e_{H_1}}\ge n$.
%Similarly, we have $v_{H_3}\ge 1$ and $n^{v_{H_3}}p^{e_{H_3}}\ge n$.
%At last, $e_{T_1\cap T_3}>0$ implies that $e_{H_2}>0$.
%At last, since $Q_1$ and $S_B$ share at least one edge, we know that $e_{H_2}>0$ or $e_{H_3}>0$.
So we have
\begin{equation}\label{eq:estnp}
n^{v_{H_1}}p^{e_{H_1}}\cdot n^{v_{H_2}}p^{e_{H_2}}\cdot n^{v_{H_3}}p^{e_{H_3}} \ge n^{3}.
\end{equation}

Now we define
$\Delta_{H_1, H_2, H_3}= \sum_{A, B} \mathbb{E}[I_A I_B]$,
%$$
%\Delta_{t_1, t_2, t_3}= 
%\sum_{\substack{ S_i, S_j: \quad |Q_1\cap Q_2|=t_1 \\ 
%		\hspace{2.03cm}|(Q_1\cup Q_2)\cap Q_3|=t_2 \\
%		\hspace{2,63cm} |(Q_1\cup Q_2\cup Q_3)\cap P^4|=t_3 }} \mathbb{E}[I_i' I_j'].
%$$
where the sum is over the pairs $\{A, B\}$ with $A\cap B\neq \emptyset$ that generate $H_1, H_2, H_3$. 
Observe that the sum contains at most
\[
\binom n{4b-v_{H_1} - v_{H_2} -v_{H_3}} (4b-v_{H_1} - v_{H_2} -v_{H_3})_{b}^4 < n^{4b-(v_{H_1} + v_{H_2} + v_{H_3})} (4b)^{3b}
\]
terms.
%So $\Delta_\cQ = \sum_{H_1, H_2, H_3} \Delta_{H_1, H_2, H_3}$.
%For fixed $H_1, H_2, H_3$, observe that any such $A\cup B$ contains $4b-(v_{H_1} + v_{H_2} + v_{H_3})$ vertices and $4a - (e_{H_1} + e_{H_2} + e_{H_3})$ edges.
%Moreover, there are at most $(4b)!2^{2b}(2b)!$ ways to label these vertices to determine $\{A, B\}$.
Thus, from~\eqref{eq:estnp}, we obtain
\begin{align*}
\Delta_{H_1, H_2, H_3}= 
\sum_{A, B} \mathbb{E}[I_A I_B] \le (4b)^{3b} n^{4b-(v_{H_1} + v_{H_2} + v_{H_3})} p^{4a - (e_{H_1} + e_{H_2} + e_{H_3})}
 \le (4b)^{3b} n^{4b - 3} p^{4a}.
\end{align*}
Let $D=D(b,k,r)$ be the number of choices for $H_1, H_2, H_3$, thus
\[
\Delta_Y = \sum_{H_1, H_2, H_3} \Delta_{H_1, H_2, H_3} \le D (4b)^{3b} n^{4b-3}p^{4a}.
\]
Therefore, by~\eqref{eq:1} and the fact that $n$ is large enough, we get
\begin{align*}
\mathbb{P}\big(Y\ge 4b^2 n^{2b-1}p^{2a}) &\le \mathbb{P}\big(Y\ge 2\mathbb{E}[Y]) \le \frac{\Delta_Y}{\mathbb{E}[Y]^2} \le \frac{D(4b)^{3b} n^{4b-3}p^{4a}}{({n}^{{2b-1}} p^{2a}/2)^2} \le \frac1{\sqrt{n}}.
\end{align*}
This verifies~\ref{item-IV}.
\end{proof}

For $m\ge k+r-1$, denote by $P_{m}^{k,r}$ the $r$\th{} power of a
$k$-uniform tight path on $m$ vertices.  Similarly,
write~$C_{m}^{k,r}$ for the $r$\th{} power of a $k$-uniform tight
cycle on $m$ vertices.  For simplicity we say that $P_{m}^{k,r}$ is an
\emph{$(r,k)$-path} and~$C_{m}^{k,r}$ is an \emph{$(r,k)$-cycle}.  We
write $P_m^r$ for $P_{m}^{k,r}$ whenever~$k$ is clear from the
context.  Moreover, the ends of $P_m^r$ are its first and last $k+r-1$
vertices (with the order in the $(r,k)$-path).  We end this section by
computing $\Phi_{P_b^r}$ for the
$p=p(n)\ge n^{-\binom{k+r-2}{k-1}^{-1}-\epsilon}$ as in
Theorem~\ref{main}.
For $b\geq k+r-1$, let
\[
g(b):= \left( b - \frac{(k-1)(k+r-1)}{k} \right) \binom{k+r-2}{k-1}.
\]
Clearly $g$ is an increasing function.
Note that the number of edges in $P_m^{k,r}$ is given by
\begin{align*}
\big|E\big(P_m^{k,r}\big)\big|&=\binom{k+r-1}{k} + \left(m-(k+r-1)\right)\binom{k+r-2}{k-1}\\
			 & = \left( m - \frac{(k-1)(k+r-1)}{k} \right) \binom{k+r-2}{k-1} = g(m).
\end{align*}

\begin{prop}\label{prop:phiP}
  Suppose~$k\ge 2$, $r\geq 1$, $b\geq k+r-1$, $k+r\ge 4$ and $C>0$.  Let
  $\epsilon$ be such that
  $0<\epsilon<\min\big\{(2g(b))^{-1},\big(3\binom{k+r-1}{k}\big)^{-1}\big\}$.
  Suppose $1/n\ll1/C,\, 1/k,\, 1/r,\, 1/b$.  If
  $p=p(n)\ge n^{-\binom{k+r-2}{k-1}^{-1}-\epsilon}$, then
  $\Phi_{P_b^r} \ge C n$.
\end{prop}
\begin{proof}
Let $H$ be a subgraph of $P_b^r$.
Since for any integer $k+r-1\le b'\le b$, any subgraph of $P_{b'}^r$ has at most $g(b')$ edges, we have the following observations.
\begin{enumerate}[label={(\alph*)}]
\item If $e_H> g(b')$ for some $b'\ge k+r-1$, then $v_H\ge b'+1$;\label{eq-edgea}
\item if $e_H > \binom{i}{k}$ for some $k-1\le i< k+r-1$, then $v_H\ge i+1$.\label{eq-edgeb}
\end{enumerate}

By~\ref{eq-edgea}, we have
\[
\min_{g(k+r-1)< e_H\le g(b)} n^{v_H} p^{e_H} = \min_{k+r-1\le b'< b}\left( \min_{g(b')< e_H\le g(b'+1)} n^{v_H} p^{e_H}\right) \ge \min_{k+r-1\le b'< b} n^{b'+1} p^{g(b'+1)}.
\]
Since $p\ge n^{-1/\binom{k+r-2}{k-1}-\epsilon}$, and $g(b'+1) >0$, the following holds for any $b'<b$:
\begin{multline*}
  \qquad
  n^{b'+1}p^{g(b'+1)}\ge
  n^{b'+1}\(n^{-1/\binom{k+r-2}{k-1}-\epsilon}\)^{g(b'+1)} \\
  = n^{-g(b'+1)\epsilon}n^{(k-1)(k+r-1)/k} \ge
  n^{-g(b)\epsilon}n^{(k-1)(k+r-1)/k}\ge C n,
  \qquad
\end{multline*}
where we used $(k-1)(k+r-1)/k\ge 3/2$ and $g(b)\epsilon< 1/2$.
Therefore, 
\begin{equation}\label{eq:estimate1}
\min_{g(k+r-1)< e_H\le g(b)} n^{v_H} p^{e_H} \geq C n.
\end{equation}
On the other hand, noting that $g(k+r-1)=\binom{k+r-1}{k}$, by~\ref{eq-edgeb} we have
\[
\min_{0< e_H\le g(k+r-1)} n^{v_H} p^{e_H} = \min_{k-1\le i< k+r-1} \left(\min_{\binom{i}{k}< e_H\le \binom{i+1}{k}} n^{v_H} p^{e_H} \right)\ge \min_{k-1\le i< k+r-1} n^{i+1} p^{\binom{i+1}{k}}.
\]
Since $p\ge n^{-1/\binom{k+r-2}{k-1}-\epsilon}$, and $\binom{i+1}{k}\epsilon\leq  1/3$ for any $k-1\le i\le k+r-2$, if $i\ge 2$, then
\[n^{i+1} p^{\binom{i+1}{k}} \ge n^{i+1} n^{-(1/\binom{k+r-2}{k-1}+\epsilon)\frac{i+1}{k}\binom{i}{k-1}} \ge n^{i+1 - \frac{i+1}{k}-\binom{i+1}{k}\epsilon}\ge C n.
\]
Otherwise $i=1$ and thus $k=2$, in which case we have $n^{i+1} p^{\binom{i+1}{k}}= n^2 p \ge C n$.
Therefore,
\begin{equation}\label{eq:estimate2}
\min_{0< e_H\le g(k+r-1)} n^{v_H} p^{e_H} \geq C n.
\end{equation}
From \eqref{eq:estimate1} and \eqref{eq:estimate2}, we have
$\Phi_{P_b^r} \ge  C n$, as desired.
\end{proof}

\section{The Connecting and Absorbing Lemmas}
\label{sec:abs-con}
For brevity, throughout the rest of this paper, we write
\begin{equation*}
h:=k+r-1\qcomma
t:=g(2h)\qcomma
c:=\binom{k+r-2}{k-1}^{-1}.
\end{equation*}
Recall that the ends of an $(r,k)$-path are ordered $h$-sets that 
span a copy of $K_h^{(k)}$ in~$H$.

\subsection{The Connecting Lemma}
Given a $k$-graph $H$ and two ordered $h$-sets of vertices $A$ and $B$
each spanning a copy of $K_h^{(k)}$ in $H$, we say that an ordered
$2h$-set of vertices $C$ \textit{connects}~$A$ and~$B$ if
$C\cap A=C\cap B=\emptyset$ and the concatenation
$\text{\textit{ACB}}$ spans a labelled copy of~$P_{4h}^r$.
We are now ready to state our connecting lemma.

\begin{lemma}[Connecting Lemma]\label{lm:conn}
  Suppose $1/n \ll\epsilon \ll \beta \ll \alpha'\ll 1/k, 1/r$.
  Let~$H$ be an $n$-vertex $k$-graph with
  $\delta_{k-1}(H)\ge \left( 1- c + \alpha' \right) {n}$ and suppose
  $p=p(n)\ge n^{-c-\epsilon}$.  Then a.a.s.\
  $H\cup \mathbb{G}^{(k)}(n,p)$ contains a set $\mathcal{C}$ of
  vertex-disjoint copies of $P_{2h}^r$ with $|\mathcal{C}|\le \beta n$
  such that, for every pair of disjoint ordered $h$-sets spanning
  a copy of $K_h^{(k)}$ in $H$, there are at least $\beta^2 n/(2h)^2$
  ordered copies of $P_{2h}^r$ in $\mathcal C$ that connect them.
\end{lemma}

\begin{proof}
Let $\cS$ be the set of pairs of disjoint ordered $h$-sets that each span a copy of $K_h^{(k)}$ in~$H$.
Fix $\{S, S'\}\in \cS$ and write $S:=(v_1, \dots, v_{h})$ and $S':=(w_{h}, \dots, w_1)$.
Since $\delta_{k-1}(H)\ge \left( 1- c + \alpha' \right) {n}$, we can extend $S$ to an $(r,k)$-path with vertices $(v_1, \dots, v_{2h})$ such that the vertices of this $(r,k)$-path are disjoint with $\{w_{h}, \dots, w_1\}$ and there are at least $(\alpha' n/2)^{h}$ choices for the ordered set $(v_{h+1}, \dots, v_{2h})$.
Similarly, we can extend $S'$ to an $(r,k)$-path $(w_{2h}, \dots, w_1)$ such that the vertices of this $(r,k)$-path are disjoint with $\{v_1, \dots, v_{2h}\}$ and there are at least $(\alpha' n/2)^{h}$ choices for the ordered set $(w_{2h}, \dots, w_{h+1})$. 
So there are at least $(\alpha' n/2)^{2h} \ge 24\beta n^{2h}$ choices for the ordered $2h$-sets
$(v_{h+1}, \dots, v_{2h}, w_{2h}, \dots, w_{h+1})$.
Let $\mathcal C_{S, S'}$ be a collection of exactly $24\beta n^{2h}$ such ordered $2h$-sets of vertices.
Clearly if an ordered set~$C$ in $\mathcal C_{S, S'}$ spans a copy of $P_{2h}^r$, then $C$ connects $S$ and $S'$.

Now we will use the edges of $G=\mathbb{G}^{(k)}(n, p)$ to obtain the desired copies of $P^r_{2h}$ that connect the pairs in $\cS$.
Let $\cT$ be the set of all labelled copies of $P_{2h}^r$ in $G$.
We claim that the following properties hold with probability at least
$1-3/\sqrt n$: 
\begin{enumerate}[label={(\alph*)}]
\item\label{item:a} $|\cT|\le 2 p^{t} n^{2h} $;
\item\label{item:b} for every $\{S,S'\}\in \cS$, at least $12\beta p^{t} n^{2h}$ members of $\cT$ connect $S$ and $S'$;
\item\label{item:c} the number of overlapping pairs of members of $\cT$ is at most $4(2h)^2 p^{2t}n^{4h-1}$.
\end{enumerate}

To see that the claim above holds, note that by
Proposition~\ref{prop:phiP}, we can apply Lemma~\ref{lm:gnp} with
$F=P_{2h}^r$, $\gamma=24\beta$ and $\mathcal C_{S, S'}$ in place of
$\cF_i$.  Items~\ref{item:a}, \ref{item:b} and~\ref{item:c} follow,
respectively, from Lemma~\ref{lm:gnp}~\ref{item-III}, \ref{item-II}
and~\ref{item-IV}.

Next we select a random collection $\mathcal{C'}$ by including each
member of $\cT$ independently with probability
$q:=\beta/(2(2h)^2 n^{2h-1}p^{t})$.  By using Chernoff's inequality
(for~\ref{item:i} and~\ref{item:ii} below) and Markov's inequality
(for~\ref{item:iii} below), we know that there is a choice of
$\mathcal C'$ that satisfies the following properties:
\begin{enumerate}[label={(\roman*)}]
\item\label{item:i}$|\mathcal C'|\le 2q |\cT|\le \beta n$;
\item\label{item:ii} for every $\{S,S'\}\in \cS$, there are at least
$12\beta (q/2) n^{2h}p^{t}  = 3 \beta^2 n/(2h)^2$
members of $\mathcal C'$ that connect $S$ and $S'$;
\item\label{item:iii} the number of overlapping pairs of members of $\mathcal C'$ is at most $ 8(2h)^2q^2 n^{4h-1} p^{2t}= 2\beta^2 n/(2h)^2$.
\end{enumerate}
Deleting one member from each overlapping pair, we obtain a collection
$\mathcal C$ of vertex disjoint copies of $P_{2h}^r$ with
$|\mathcal C|\le \beta n$, and such that, for every pair of disjoint ordered
$h$-sets each spanning a~$K_h^{(k)}$ in~$H$, there are at
least $3\beta^2 n/(2h)^2 - 2\beta^2 n/(2h)^2= \beta^2 n/(2h)^2$ sets
of~$2h$ vertices connecting them.
\end{proof}

\subsection{The Absorbing Lemma}
In this subsection we prove our absorbing lemma.

\begin{lemma}[Absorbing Lemma]\label{lm:abs}
  Suppose $1/n\ll\epsilon \ll \zeta \ll \alpha\ll 1/k,1/r$.  Let~$H$
  be an $n$-vertex $k$-graph with
  $\delta_{k-1}(H)\ge \left( 1- c + \alpha \right) {n}$ and suppose
  $p=p(n)\ge n^{-c-\epsilon}$.  Then a.a.s.
  $H\cup \mathbb{G}^{(k)}(n, p)$ contains an $(r,k)$-path $P_{\abs}$
  of order at most $6h \zeta n$ such that, for every set
  $X\subseteq V(H)\setminus V(P_{\abs})$ with
  $|X|\leq \zeta^2n/(2h)^2$, there is an $(r,k)$-path in $H$ on
  $V(P_{\abs})\cup X$ that has the same ends as~$P_{\abs}$.
\end{lemma}

We call the $(r,k)$-paths~$P_{\abs}$ in Lemma~\ref{lm:abs} \emph{absorbing
  paths}.  We now define \emph{absorbers}.

\begin{dfn}
  Let $v$ be a vertex of a $k$-graph.  An ordered $2h$-set of
  vertices $(w_1,\dots ,w_{2h})$ is a \emph{$v$-absorber} if
  $(w_1,\dots,w_{2h})$ spans a labelled copy of~$P_{2h}^r$ and
  $(w_1,\dots,w_{h},v,w_{h+1},\dots,w_{2h})$ spans a
  labelled copy of~$P_{2h+1}^r$.
\end{dfn}

\begin{proof}[Proof of Lemma~\ref{lm:abs}]
  Suppose $1/n\ll\epsilon \ll \zeta\ll \beta \ll \alpha\ll 1/k,\,1/r$.
  We split the proof into two parts. We first find a set $\mathcal{F}$
  of absorbers and then connect them to an $(r,k)$-path by using
  Lemma~\ref{lm:conn} (Connecting Lemma).  We will expose
  $G= \mathbb{G}^{(k)}(n, p)$ in two rounds: $G=G_1\cup G_2$
  with~$G_1$ and $G_2$ independent copies 
  of~$\mathbb{G}^{(k)}(n,p')$, where
  $(1-p')^2 = 1-p$.

Fix a vertex $v$.
By the codegree condition of $H$, we can extend $v$ to a labelled copy of~$P_{2h+1}^r$ in the form
$(w_1, \dots, w_{h}, v, w_{h+1},\dots, w_{2h})$
such that  there are at least $(\alpha n/2)^{2h} \ge 24\zeta n^{2h}$ choices for the ordered $2h$-set $(w_{1}, \dots, w_{2h})$.
Let $\mathcal A_{v}$ be a collection of exactly $24\zeta n^{2h}$ such ordered $2h$-sets.
By definition, if an ordered set $A$ in $\mathcal A_{v}$ spans a labelled copy of $P_{2h}^r$, then~$A$ is a $v$-absorber.

Now consider $G_1=\mathbb{G}^{(k)}(n, p')$ and let $\cT$ be the set of all labelled copies of $P_{2h}^r$ in $G_1$.
By Proposition~\ref{prop:phiP}, we can apply Lemma~\ref{lm:gnp} with $F=P_{2h}^r$ and $\cA_v$ in place of $\cF_i$.
Using the union bound we conclude that the following properties hold
with probability at least $1-3/\sqrt n$:
\begin{enumerate}[label={(\alph*)}]
%\item every induced subgraph of $H_1$ of order $\beta n$ contains a copy of $P_m^r$;\label{propa}
\item $|\cT|\le 2 p^{t} n^{2h} $;\label{propc}
\item for every vertex $v$ in $H$, at least $12\zeta p^{t} n^{2h}$ members of $\cT$ are $v$-absorbers;\label{propb}
\item the number of overlapping pairs of members of $\cT$ is at most $4(2h)^2 p^{2t}n^{4h-1}$.\label{propd}
\end{enumerate}

Next we select a random collection $\mathcal{F}'$ by including each member of $\cT$ independently with probability $q=\zeta/(2(2h)^2 p^{t} n^{2h-1})$. 
%By~\ref{propc}, we have $\mathbb{E}[|\mathcal{F'}|]\leq q\cdot 2p^{t}n^{2h} = \zeta n$.
In view of the properties above, by using Chernoff's inequality
(for~\ref{item:i.2} and~\ref{item:ii.2} below) and Markov's inequality
(for~\ref{item:iii.2} below), we know that there is a choice of
$\mathcal F'$ that satisfies the following properties:
\begin{enumerate}[label={(\roman*)}]
\item\label{item:i.2} $|\mathcal F'|\le \zeta n$;
\item\label{item:ii.2} for every vertex $v$, at least
$12\zeta (q/2) p^{t} n^{2h} = 3\zeta^2 n/(2h)^2$
members of $\mathcal F'$ are $v$-absorbers;
\item\label{item:iii.2} there are at most $8(2h)^2 q^2 n^{4h-1} p^{2t}= 2\zeta^2 n/(2h)^2$ overlapping pairs of members of $\mathcal F'$.
\end{enumerate}
By deleting from $\cF'$ one member from each overlapping pair and all members that are not in $\cT$, we obtain a collection $\cF$ of vertex-disjoint copies of $P_{2h}^r$ such that $|\cF|\le \zeta n$, and for every vertex~$v$, there are at least
$3\zeta^2 n/(2h)^2 - 2\zeta^2 n/(2h)^2= \zeta^2 n/(2h)^2$ $v$-absorbers.
%The properties (2) and (3) hold by construction and for (4) we recall that $v$-absorber satisfy (3) by definition. Therefore set $\mathcal{F}$ has all the desired properties. 
%\end{proof}

Now we connect these absorbers using Lemma~\ref{lm:conn}.
Let $V'=V(H)\setminus V(\mathcal{F})$ and $n'=|V'|$. In particular, $n'\geq n/2$ is sufficiently large. 
Now consider $H'=H[V']$ and $G'=G_2[V'] = \mathbb{G}^{(k)}(n',p')$. 
Since $|V(\mathcal{F})|\leq  2h\cdot \zeta n \leq \alpha^2 n$, we have
$\delta_{k-1}(H')\ge \left( 1- c + \alpha/2 \right) {n}$. 
We apply Lemma~\ref{lm:conn} on $H'$ and $G'$ with $\alpha'=\alpha/2$ and $\beta$, and conclude that \aas~$H'\cup G'$ contains a set $\mathcal{C}$ of vertex-disjoint copies of $P_{2h}^r$ such that $|\mathcal{C}|\leq \beta n$ and for every pair of ordered $h$-sets in $V'$, there are at least $\beta^2 n$ members of $\mathcal{C}$ connecting them.

For each copy of $P_{2h}^r$ in $\mathcal F$, we greedily extend its two ends by $h$ vertices such that all new paths are pairwise vertex disjoint and also vertex disjoint from $V(\mathcal C)$. 
This is possible because of the codegree condition of $H_0$ and $|V(\mathcal{F})| + 2h|\mathcal{F}| + |V(\mathcal C)|\leq  2h \zeta n + 2h \zeta n + 2h\cdot \beta n< \alpha n/4$.
Note that both ends of these $(r,k)$-paths $P_{4h}^r$ are in $V'\setminus V(\mathcal C)$.
Since $\zeta n\le \beta^2 n'/(2h)^2$, we can greedily connect these $P_{4h}^r$.
Let $P_{\abs}$ be the resulting $(r,k)$-path.
By construction, $|V(P_{\abs})|\le (4h + 2h)\cdot \zeta n = 6h \zeta n$.
% \[
% |V(P_{\abs})|\le (4h + 2h)\cdot \zeta n = 6h \zeta n.
% \]
Moreover, for any $X\subseteq V\setminus V(P_{\abs})$ such that $|X|\leq \zeta^2n/(2h)$, since each vertex $v$ has at least $\zeta^2 n/(2h)^2$ $v$-absorbers in $\mathcal F$, we can absorb them greedily and conclude that there is an $(r,k)$-path on $V(P_{\abs})\cup X$ that has the same ends as $P_{\abs}$.
%So $(\dagger)$ follows.
%At last, $(\ddagger)$ follows from \ref{propa}.
\end{proof}

\section{Proof of Theorem~\ref{main}}\label{sec:main}

We now combine Lemmas~\ref{lm:conn}~and~\ref{lm:abs} to prove
Theorem~\ref{main}.

\begin{proof}[Proof of Theorem~\ref{main}]
Suppose $1/n \ll \epsilon\ll \beta \ll \zeta\ll \alpha, 1/k, 1/r$.
Furthermore, recall that $c:=\binom{k+r-2}{k-1}^{-1}$ and suppose $H\cup \mathbb{G}^{(k)}(n, p)$ is an $n$-vertex $k$-graph with 
$\delta_{k-1}(H)\ge \left( 1- c + \alpha \right) {n}$
and $p=p(n)\ge n^{- c - \epsilon}$.
We will expose $G:=\mathbb{G}^{(k)}(n, p)$ in three rounds: $G=G_1\cup
G_2\cup G_3$ with~$G_1$, $G_2$ and~$G_3$ three independent copies of
$\mathbb{G}^{(k)}(n, p')$, where $(1-p')^3=1-p$.
Note that $p' >p/3 > n^{-c-2\epsilon}$.

By Lemma~\ref{lm:abs} with $2\epsilon$ in place of $\epsilon$, \aas~the $k$-graph $H\cup G_1$ contains an absorbing $(r,k)$-path $P_{\abs}$ of order at most $6h \zeta n$, that is, for every set $X\subseteq V(H)\setminus V(P_{\abs})$ such that $|X|\leq \zeta^2n/(2h)^2$, there is an $(r,k)$-path in $H$ on $V(P_{\abs})\cup X$ which has the same ends as~$P_{\abs}$. 
%satisfying Lemma~\ref{lm:abs}~$(\dagger)$ and $(\ddagger)$ with $m=g^{-1}(1/(4\epsilon))$.
%Note that $m\ge 1/\sqrt{\epsilon}$ because $\epsilon$ is small enough and $g$ is linear.
Let $V'= V(H)\setminus V(P_{\abs})$ and $n'=|V'|$. In particular, $n'\ge (1-6h\zeta)n$ and, since $\zeta$ is small enough, we have $(n')^{c+\epsilon} \ge n^{c+\epsilon}/2$.
Thus $p' > p/2 \ge n^{-c-\epsilon}/2 \ge (n')^{-c-\epsilon}/4 \ge (n')^{-c-2\epsilon}$.
% \[
% p' > p/2 \ge n^{-c-\epsilon}/2 \ge (n')^{-c-\epsilon}/4 \ge (n')^{-c-2\epsilon}.
% \]

Now consider $H'=H[V']$ and let $G_2':=\mathbb{G}^{(k)}(n', p')$ be the subgraph of $G_2$ induced by~$V'$.
Note that $\delta_{k-1}(H')\ge \delta_{k-1}(H) - |V(P_{\abs})|\ge \left( 1- c + \alpha/2 \right) {n'}$.
By Lemma~\ref{lm:conn}, \aas~the $k$-graph $H'\cup G_2'$ contains a set $\mathcal{C}$ of vertex-disjoint copies of $P_{2h}^r$ such that $|\mathcal{C}|\le \beta n$ and for every pair of disjoint ordered $h$-sets in $V'$ that each spans a copy of $K_h^{(k)}$, there are at least
$\beta^2 n'/(2h)^2$ members of $\mathcal C$ connecting them.
Since $|V(\mathcal C)|+|V(P_{\abs})|\le 2h\cdot \beta n+ 6h \zeta n\le \alpha n/2$, we can greedily extend the two ends of $P_{\abs}$ by $h$ vertices so that the two new ends $E_1, E_2$ are in $V'\setminus V(\mathcal C)$.

Let $m:=g^{-1}(1/(2\epsilon))$.
Note that $m\ge 1/\sqrt{\epsilon}$ because $\epsilon$ is small enough and $g$ is linear.
By Proposition~\ref{prop:phiP}, we can apply Lemma~\ref{lm:gnp}~\ref{item-I} with $b=m$ on $G_3$ and conclude that \aas~every induced subgraph of $G_3$ of order $\beta n$ contains a copy of $P_m^r$.
Thus we can greedily find at most $\sqrt\epsilon n$ vertex-disjoint copies of $P_m^r$ in $V'\setminus (V(\mathcal C)\cup E_1\cup E_2)$, which together covers all but at most $\beta n$ vertices of $V'\setminus V(\mathcal C)$.
Since $\sqrt\epsilon n+1\le \beta^2 n'/(2h)^2$, we can greedily connect these $(r,k)$-paths $P_m^r$ and $P_{\abs}$ to an $(r,k)$-cycle $Q^r$.
Let $R:=V(H)\setminus V(Q^r)$ and
note that $|R|\le |V(\mathcal C)|+\beta n\le (2h+1) 2\beta n\le \zeta^2 n/(2h)^2$.
Since $P_{\abs}$ is an absorber, there is an $(r,k)$-path on $V(P_{\abs})\cup R$ which has the same ends as $P_{\abs}$.
So we can replace $P_{\abs}$ by this $(r,k)$-path in $Q^r$ and obtain the $r$\th{} power of a tight Hamilton cycle.

Moreover, since all previous steps can be achieved \aas, by the union bound, $H\cup G$ \aas~contains the desired $r$\th{} power of a tight Hamilton cycle.
\end{proof}

\section{Concluding remarks}
\label{sec:concluding}

Let us briefly discuss the hypotheses in Theorem~\ref{main}.  Note
that, for~$r=1$, the condition in~\eqref{eq:main_min_deg} is simply
$\delta_{k-1}(H)\geq\alpha n$, with~$\alpha$ any arbitrary positive
constant.  Thus, in this case, our theorem is in the spirit of the
original Bohman, Frieze and Martin~\cite{BFM} set-up, in the sense
that we have a similar minimum degree condition on the deterministic
graph~$H$.  However, if~$r>1$, then our minimum
condition~\eqref{eq:main_min_deg} is of the form
$\delta_{k-1}(H)\geq(\sigma+\alpha)n$ for some~$\sigma=\sigma(k,r)>0$
(and arbitrarily small~$\alpha>0$).  Thus, for~$r>1$, our result is
more in line with Theorem~\ref{thm:BDF} of Bennett, Dudek and
Frieze~\cite{BeDuFr16} (in fact, we have~$\sigma(2,2)=1/2$ in our
result, which matches the minimum degree condition in
Theorem~\ref{thm:BDF}).  It is natural to ask whether one can weaken
the condition in~\eqref{eq:main_min_deg}
to~$\delta_{k-1}(H)\geq\alpha n$, that is, whether one can
have~$\sigma=0$.  This problem was settled positively by B\"ottcher,
Montgomery, Parczyk and Person for graphs~\cite{BMPP}.  However, the
problem remains open for $k$-graphs ($k\geq3$).

\begin{question}
  \label{prog:main}
  Let integers $k\ge3$ and $r\ge2$ and $\alpha>0$ be given.  Is there
  $\epsilon>0$ such that, if~$H$ is a $k$-graph on~$n$ vertices
  with~$\delta_{k-1}(H)\ge\alpha n$
  and $p=p(n)\ge n^{-\binom{k+r-2}{k-1}^{-1}-\epsilon}$, then~\aas\
  $H\cup \mathbb{G}^{(k)}(n,p)$ contains the $r$\th{} power of a
  tight Hamilton cycle?
\end{question}

Two remarks on the value of~$\sigma=\sigma(k,r)$ in our degree
condition~\eqref{eq:main_min_deg} follow.  These remarks show that,
even though~$\sigma>0$ if~$r>1$, the value of~$\sigma$ is (in the
cases considered) below the value that guarantees that~$H$ on its own
contains the $r$\th{} power of a tight Hamilton cycle.

Let us first consider the case~$k=2$, that is, the case of graphs.  In
this case, $\sigma=1-1/r$ and condition~\eqref{eq:main_min_deg}
is~$\delta(H)\geq(1-1/r+\alpha)n$.  We observe that this condition
does \textit{not} guarantee that~$H$ contains the $r$\th{} power of a
Hamilton cycle; the minimum degree condition that does is
$\delta(H)\geq(1-1/(r+1))n=rn/(r+1)$, and this value is optimal.

Let us now consider the case $k=3$ and $4\mid n$.  In this case, a
construction of Pikhurko~\cite{Pik08} shows that the
condition~$\delta_2(H)\geq3n/4$ does not guarantee the existence of
the square of a tight Hamilton cycle in~$H$ (in fact, his
constructions is stronger and shows that this condition does not
guarantee a $K_4^{(3)}$-factor in~$H$).  Our minimum degree condition
for~$k=3$ and~$r=2$ is~$\delta_2(H)\geq(2/3+\alpha)n$.

Finally, a simple calculation shows that the expected number
of~$P_n^r$ in~$\GG^{(k)}(n,p)$ is~$o(1)$
if~$p\leq\big((1-\epsilon)e/n\big)^{\binom{k+r-2}{k-1}^{-1}}$
and~$\epsilon>0$.  Thus, for such a~$p$, a.a.s.~$\GG^{(k)}(n,p)$ does
\textit{not} contain the $r$\th{} power of a tight Hamilton cycle.

\renewcommand{\doitext}{DOI\,}
\renewcommand{\PrintDOI}[1]{\doi{#1}}
\renewcommand{\eprint}[1]{\href{http://arxiv.org/abs/#1}{arXiv:#1}}

\endgroup
\end{document}